\documentclass[12pt]{elsarticle} 
\usepackage[T1]{fontenc}
\usepackage[utf8]{inputenc}
\usepackage[english]{babel}
\usepackage{amsmath,amsthm,amssymb}
\usepackage{natbib}

\usepackage{algorithm}
\usepackage{algpseudocode}
\usepackage{graphicx}
\usepackage{subcaption}
\usepackage{placeins}

\newtheorem{theorem}{Theorem}
\newtheorem{proposition}[theorem]{Proposition}
\newtheorem{definition}[theorem]{Definition}
\newtheorem{lemma}[theorem]{Lemma}
\newtheorem{corollary}[theorem]{Corollary}

\theoremstyle{remark}
\newtheorem{remark}[theorem]{Remark}
\newtheorem{example}[theorem]{Example}

\newcommand{\mat}[1]{{\boldsymbol{#1}}}
\newcommand{\vect}[1]{{\boldsymbol{#1}}}
\newcommand{\tp}{\mathrm{T}}
\newcommand{\herm}{\mathrm{H}}

\newcommand{\rmi}{{\mathrm{i}}}
\newcommand{\rmj}{{\mathrm{j}}}
\newcommand{\rmk}{{\mathrm{k}}}

\newcommand{\rme}{{\mathrm{e}}}

\newcommand{\bv}{\vect{v}}
\newcommand{\bw}{\vect{w}}
\newcommand{\bx}{\vect{x}}
\newcommand{\by}{\vect{y}}

\newcommand{\bSigma}{\mat{\Sigma}}

\newcommand{\bLambda}{\mat{\Lambda}}

\newcommand{\bA}{\mat{A}}
\newcommand{\bB}{\mat{B}}

\newcommand{\bD}{\mat{D}}

\newcommand{\bG}{\mat{G}}

\newcommand{\bI}{\mat{I}}

\newcommand{\bM}{\mat{M}}

\newcommand{\bP}{\mat{P}}
\newcommand{\bQ}{\mat{Q}}
\newcommand{\bV}{\mat{V}}

\newcommand{\bU}{\mat{U}}
\newcommand{\bX}{\mat{X}}
\newcommand{\bY}{\mat{Y}}

\newcommand{\bR}{\mat{R}}
\newcommand{\bE}{\mat{E}}

\newcommand{\setR}{{\mathbb R}}
\newcommand{\setH}{{\mathbb H}}
\newcommand{\setC}{{\mathbb C}}
\newcommand{\setZ}{{\mathbb Z}}

\newcommand{\setA}{{\mathbb A}}

\newcommand{\setK}{{\mathbb K}}

\DeclareMathOperator{\trace}{tr}

\DeclareMathOperator{\sign}{sgn}
\DeclareMathOperator{\mathcalV}{\mathcal{V}}

\DeclareMathOperator*{\argmax}{arg\,max}

\journal{Linear Algebra and its Applications}

\begin{document}
\begin{frontmatter}
	
\title{The QRD and SVD of matrices over a real algebra \tnoteref{star}}

\author{Paul Ginzberg\corref{mycorrespondingauthor}}
\cortext[mycorrespondingauthor]{Corresponding author}
\ead{paul.ginzberg05@imperial.ac.uk}

\author{Christiana Mavroyiakoumou\corref{cor1}}
\ead{christiana.mavroyiakoumou13@imperial.ac.uk}
\address{Department of Mathematics, Imperial College London, South Kensington Campus, SW7 2AZ London, UK}
\tnotetext[star]{This work was supported by an Undergraduate Research Opportunity Program funding from the Department of Mathematics, Imperial College London.}

\begin{abstract}
Recent work in the field of signal processing has shown that the singular value decomposition of a matrix with entries in certain real algebras can be a powerful tool. 
In this article we show how to generalise the QR decomposition and SVD to a wide class of real algebras, including all finite-dimensional semi-simple algebras, (twisted) group algebras and Clifford algebras. Two approaches are described for computing the QRD/SVD: one Jacobi method with a generalised Givens rotation, and one based on the Artin-Wedderburn theorem.
\end{abstract}

\end{frontmatter}

\section{Introduction}

The singular value decomposition (SVD) is one of the most commonly used matrix decompositions. In particular, the eigenvalue decomposition (EVD) of a symmetric positive semi-definite matrix (e.g. a covariance matrix) is equal to its SVD. Although the SVD is most commonly applied to real-valued matrices, SVDs of complex-valued matrices are common in signal processing, with the complex number $r \rme^{i \phi}$ representing a narrowband wave with amplitude $r$ and phase $\phi$. Recently in the signal processing literature, generalisations of the SVD/EVD to some other algebras have been suggested, most notably the algebra of Laurent polynomials $\setR[z,z^{-1}]$  \citep{mcwhirter_evd_2007,foster_algorithm_2010} and the algebra of quaternions $\setH$ \citep{miron_quaternion-music_2006}, but also
biquaternions $\setH \otimes \setC$ \citep{le_bihan_music_2007}, reduced quaternions  \citep{gai_denoising_2015,gai_reduced_2014}, quaternion Laurent polynomials $\setH[z,z^{-1}]$ \citep{menanno_quaternion_2010}, and quad-quaternions $\setH \otimes \setH$ \citep{gong_quad-quaternion_2008,xiao_direction_2014}.



 In this article we propose two methods for computing the SVD (and QRD) over a general real algebra $\setA$. The first method generalises the approach of \citet{foster_algorithm_2010} and can be applied to many finite and infinite-dimensional real algebras, including all (twisted) group algebras. The second approach applies to finite-dimensional semi-simple algebras and uses an appropriate representation of the algebra as in \citet{gai_denoising_2015} to reduce the problem to parallel real, complex or quaternion SVDs. Although both approaches have mild technical conditions on the algebra, these conditions are satisfied by all Clifford algebras $C\ell(p,q)$, so that the SVD of a matrix in $C\ell(p,q)^{m \times n}$ can be computed.
 
Yet another method for the computation of matrix decompositions in an algebra is given by \citet{lang_self-adaptive_2012}, which obtains the quaternion PCA of a quaternion matrix  from the complex PCA of its complex matrix representation, by computing 
a large number of
 Moore-Penrose inverses. The authors believe that this third approach is very inefficient and hence it will not be generalised or explored further in this article.

Section~\ref{sec:preliminaries} provides some preliminary definitions and assumptions. Section~\ref{sec:givens} introduces a generalised Givens rotation and uses it to define a Jacobi QR by columns algorithm (Algorithm~\ref{code:QR}). Section~\ref{sec:SVD} describes the SVD by QR algorithm (Algorithm~\ref{code:SVD}). Section~\ref{sec:semisimple} describes a second approach to computing the SVD/QRD which is based on the algebra representation from the Artin-Wedderburn theorem. Finally, Section~\ref{sec:examples} demonstrates the usefullness of our general framework by considering three particular examples of algebra: multivariate Laurent polynomials and the conformal geometric algebra, for which no previous SVD algorithm exists, and quad-quaternions, for which our theory improves the existing SVD algorithm.

\section{Preliminaries}\label{sec:preliminaries}
The following preliminary constructs and assumptions are introduced so that we may define the SVD on an algebra in a way which relates simply to the usual real SVD.
\begin{definition}
$\setA$ is a $d$-dimensional algebra over the field $\setK$ if it is a $d$-dimensional vector space over $\setK$,
with a multiplication satisfying $\forall x,y,z \in \setA,\ \forall r,s \in \setK$
\begin{eqnarray*}
&&x(y+z)=xy+xz,\\
&&(y+z)x=yx+zx,\\
&&(rx)(sy)=(rs)(xy).
\end{eqnarray*}

$\setA$ is associative if furthermore $x(yz)=(xy)z$.

$\setA$ is unital if it contains a multiplicative identity. 
\end{definition}

When referring to an algebra we will henceforth assume that the field is $\setK=\setR$, and that the algebra is unital and associative (assumption \textbf{A0}). In an abuse of notation we will denote all multiplicative identities by $1$, the meaning being clear from context. In particular, we will identify $r \in \setR$ with $r 1 = r \in \setA$.

A unital associative real algebra can equivalently be defined as a ring which contains $\setR$ as a subring of its centre.

Let $\setA$ be a $d$-dimensional algebra with (ordered) basis $\mathcal{B}=\{e_1,\ldots,e_d\}\subset \setA$. The basis defines a vector-space isomorphism $\mathcal{V} : \setA \rightarrow \setR^d$; namely
$
\mathcal{V}(a_1 e_1+\ldots+a_d e_d)=\left(a_1,\ldots,a_d\right)^\tp
$.
The vector-space isomorphism in turn defines an injective algebra homomorphism $\tilde{\bullet}: \setA \rightarrow \setR^{d \times d}$; namely for $a \in \setA$, $\tilde{a}$ is the unique linear transformation such that $\forall x \in \setA\ \tilde{a}\mathcal{V}(x)=\mathcal{V}(ax)$. We will henceforth refer to $\tilde{\setA}=\{ \tilde{a}:a \in\setA\} \subseteq \setR^{d \times d}$ as ``the'' real matrix representation (RMR) of $\setA$. $\tilde{\setA}$ is itself an algebra isomorphic to $\setA$, and hence we could assume without loss of generality that $\setA$ is a subalgebra of $\setR^{n \times n}$.

The notation used in this article sometimes implicitly assumes $d<\infty$ for simplicity. However the results in Sections~\ref{sec:preliminaries},\,\ref{sec:givens}\,\&\,\ref{sec:SVD} remain applicable when $d$ is countably infinite. In the case $d=\infty$ we must however assume that $\setA$ is the (finite not closed) span of its basis elements, i.e.\ every element of $\setA$ has finitely-many non-zero coefficients. Without this assumption, even simple addition may require an infinite number of operations.

We will for the rest of Sections~\ref{sec:preliminaries}\,\&\,\ref{sec:givens} make the assumption:
\begin{description}
\item[A1:] $e_1=1$. (For the RMR $\tilde{\setA}$ we have $\tilde{e_1}=\bI_d=1$ since we identify $1$ with the multiplicative identity.)
\end{description}
A1 allows us to define the real part $\Re(a_1 e_1+\ldots+a_d e_d)=a_1$. We will typically have $\Re(a)=\frac{1}{d}\trace(\tilde{a})$, and take this to be the definition of $\Re$ when the basis does not satisfy A1. The two definitions of $\Re$ are equivalent whenever $\bI_d$ is orthogonal to the other basis elements $\tilde{e_i}$ in $\setR^{d \times d}$.

Since $\setA$ is isomorphic to $\setR^d$ as a vector space, for $a \in \setA$ we may define the usual Euclidean norm $\|a\|_2=\|\mathcal{V}(a)\|_2$ and supremum norm $\|a\|_\infty=\|\mathcal{V}(a)\|_\infty$. For $\bX = (x_{ij}) \in \setA^{m \times n}$ we define the Frobenius (Euclidean) norm \linebreak $\|\bX\|_F=\left(\sum_{i=1}^m\sum_{j=1}^n \|x_{ij}\|_2^2\right)^\frac{1}{2}$ and supremum norm $\|\bX\|_\infty=\max\limits_{i,j}\|x_{ij}\|_\infty$. $\|\bullet\|$ will be used to denote an unspecified norm.
\begin{definition}
$\bar{\bullet}:\setA \rightarrow \setA$ is an involution if it satisfies $\forall x,y \in \setA$, $a \in \setR \subseteq \setA$
\begin{align*}
\overline{\left(\bar{x}\right)}&=x\\
\overline{xy}&=\bar{y}\bar{x}\\
\bar{a}&=a.
\end{align*}
An algebra with an involution is called a $*$-algebra.
\end{definition}
$\setR^{d \times d}$ is a $*$-algebra whose (standard) involution is the matrix transpose $\bullet^\tp$. In particular, $\setR$ is a $*$-algebra with the identity as its involution.
\begin{definition}
 For a real $*$-algebra $\setA$, let $\mathcal{U}(\setA)=\{ a \in \setA : \bar{a}a=1 \}$ denote the set of unitary elements.
\end{definition}
$\mathcal{U}(\setA)$ is a subset of the set of invertible elements (a.k.a.\ units), and forms a multiplicative group.
We will henceforth make the assumption:
\begin{description}
\item[A2:] $\setA$ is a $*$-algebra and $\|ba\|_2 = \|a\|_2 \ \forall a \in \setA,\ b \in \mathcal{U}(\setA)$.
\end{description}
A sufficient condition%
\footnote{It is sufficient because $\mathcal{U}(\tilde{\setA}) \subseteq \mathcal{U}(\setR^{d \times d})$.  $\mathcal{U}(\setR^{d \times d})$ is the set of $d \times d$ orthogonal matrices, and orthogonal matrices are isometries.}
 for A2 is the following assumption:
\begin{description}
\item[A3:] $\tilde{\setA}$ is closed under $\bullet^\tp$, and we endow $\setA$ with the $*$-algebra structure induced from its RMR, i.e.\ $\forall\, a \in \setA$  we define its involution $\bar{a} \in \setA$ to be the unique element satisfying $\tilde{\bar{a}}=\tilde{a}^\tp$.
\end{description}
The stronger assumption A3 makes $\tilde{\bullet}$ a $*$-algebra isomorphism between $\setA$ and $\tilde{\setA}$, so that $\setA$ is a sub-$*$-algebra of $\setR^{d \times d}$.

\begin{definition}\label{def:hermu}
For matrices $\bX \in \setA^{m \times n}$ define the Hermitian transpose as $\bX^\herm=\bar{\bX}^\tp$. A matrix $\bX \in \setA^{m \times m}$ is said to be Hermitian if $\bX = \bX^\herm$. It is said to be unitary if $\bX^\herm\bX=\bI_m$. 
\end{definition}
Note that for $\setA=\setC$ or $\setA=\setH$, Definition~\ref{def:hermu} gives us the usual definitions.

$\setA^{m \times m}$ is itself a $*$-algebra with involution $\bullet^\herm$, and the set of unitary matrices is precisely $\mathcal{U}(\setA^{m \times m})$.
\begin{definition}
A singular value decomposition (SVD) (or $\setA$SVD) of a matrix $\bX \in \setA^{m \times n}$ is a decomposition of the form $\bX=\bU \bLambda \bV^\herm$, where $\bLambda \in \setA^{m \times n}$ is diagonal, and $\bU \in \setA^{m \times m}$ and $\bV \in \setA^{n \times n}$ are unitary.
\end{definition}
\begin{definition}
A QR (or $\setA$QR) decomposition of a matrix $\bX \in \setA^{m \times n}$ is a decomposition of the form $\bX=\bQ \bR$, where $\bR \in \setA^{m \times n}$ is upper-triangular, and $\bQ \in \setA^{m \times m}$ is unitary.
\end{definition}

\section{A general class of Givens rotations}\label{sec:givens}
\subsection{The $\setA$QR by columns algorithm}
In this section we will propose a Jacobi algorithm for the QR decomposition which generalises the one by \citet{foster_algorithm_2010}. The key difference rests in describing a general approach for defining an appropriate Givens/elementary rotation for the algebra being used.

\begin{definition}
For $1 \leq i \leq m$ and $b \in \mathcal{U}(\setA)$ define the $m \times m$ shift matrix
\begin{equation*}
\bB(b,i)=\left(\begin{array}{ccc}
\bI_{i-1} &. &. \\
.& b &. \\
.&. & \bI_{m-i}
\end{array}\right).
\end{equation*}
For $1 \leq j< i \leq m$ and $b \in \mathcal{U}(\setA)$ define the $m \times m$ $\setA$-Givens rotation
\begin{eqnarray*}
\bG(\theta,b,i,j) &=&\bB(b,i)\bG(\theta,1,i,j)\bB(b,i)^\herm\\
&=&\left(\begin{array}{ccccc}
\bI_{j-1} &. &. &. &. \\
.& \cos(\theta) &. &-\sin(\theta)\bar{b} &. \\
.& .& \bI_{i-j-1} &. &. \\
.& \sin(\theta)b &. &\cos(\theta) &.\\
.& .& .&. & \bI_{m-i}
\end{array}\right),
\end{eqnarray*}
where $\bG(\theta,1,i,j)$ is the usual real Givens rotation.
\end{definition}
Every $\setA$-Givens rotation is unitary, since it is the product of a unitary $\setR$-Givens rotation $\bG(\theta,1,i,j)$ with unitary diagonal matrices. Note that $\bB(b,i)^\herm=\overline{\bB(b,i)}=\bB(\bar{b},i)$,  that $\bG(\theta,b,i,k)^\herm=\bG(-\theta,b,i,k)$ and that $\bG(0,b,i,k)=\bI_{m}$.
\begin{lemma}\label{lm:givens}
Let 
$j<i$, $\bv \in \setA^{m \times 1}$, $b \in \mathcal{U}(\setA)$, $\theta=-\mathrm{atan2}\left(\Re(\bar{b} v_i),\Re(v_j)\right)$ and $\bw=\bG(\theta,b,i,j)\bv$. 
Then 
$\Re(w_j)^2=\Re(v_j)^2+\Re(\bar{b}v_i)^2$.
\end{lemma}
\begin{proof}
If $\Re(\bar{b} v_i)=0$ then $w_j=\pm v_j$, so we may assume $\Re(\bar{b} v_i)\neq0$. 
\begin{align*}
\Re(w_j)^2&=\Re(v_j\cos(\theta)-\bar{b}v_i\sin(\theta))^2\\
&=\left(\Re(v_j)\cos(\theta)-\Re(\bar{b}v_i)\sin(\theta)\right)^2\\
&=\left(\Re(v_j)^2+\Re(\bar{b}v_i)^2\right)^{-1}\left(\Re(v_j)^2+\Re(\bar{b}v_i)^2\right)^2.
\end{align*}
\end{proof}

\begin{algorithm}[htbp]
\caption{$\setA$QR (By Columns) Decomposition}\label{code:QR}
\begin{algorithmic}[1]
\State \textbf{Input:} The matrix to be decomposed $\bA \in \setA^{m \times n}$
\State \textbf{Specify:} A function $\beta:\setA \rightarrow \mathcal{U}(\setA)$,\\ \hspace{4.5em}a norm $\|\bullet\|$,\\ \hspace{4.5em}and the error tolerance $\epsilon>0$.
\State\textbf{Initialise:} $\bQ \gets \bI_m$, $\bR \gets \bA$ and $g_1 \gets 1+\epsilon$.
\While{$g_1>\epsilon$} \label{line:outerloopstart}
\For{$k=1,\ldots,\min(m,n)$}
\State $b \gets \beta(r_{k,k})$ \label{line:extrastart}
\State $\bR \gets \bB(\bar{b},k) \bR$ \label{line:extramid}
\State $\bQ \gets \bQ \bB(b,k)$ \label{line:extraend}
\If{$k=m$} 
\State \textbf{break}
\EndIf
\While{1} \label{line:innerloopstart}
\State $i \gets \argmax\limits_{\ell:\ell>k} \|r_{\ell k}\|$
\State $g_2 \gets \|r_{ik}\|$
\If{$g_2 \leq \epsilon$} 
\State \textbf{break}
\EndIf
\State $b \gets \beta(r_{i,k})$
\State $\theta \gets \mathrm{atan2}(\Re(\bar{b }r_{i,k}),\Re(r_{k,k}))$
\State $\bR \gets \bG(-\theta,b,i,k) \bR$ \label{line:givens}
\State $\bQ \gets \bQ \bG(\theta,b,i,k)$
\EndWhile \label{line:innerloopend}
\EndFor
\State $g_1 \gets \max\limits_{\ell,j:\ell>j}\|r_{\ell j}\|$
\EndWhile \label{line:outerloopend}
\For{$k=1,\ldots,\min(m,n)$}\label{line:extrastart2}
\If{$\Re(r_{k,k})\neq0$}
\State $\bR \gets \bB(\sign(\Re(r_{k,k})),k) \bR$
\State $\bQ \gets \bQ \bB(\sign(\Re(r_{k,k})),k)$ 
\EndIf
\EndFor \label{line:extraend2}
\State \textbf{Output:} $\bR$ and $\bQ$
\end{algorithmic}
\end{algorithm}

Algorithm~\ref{code:QR} is based on \citet[Table I]{foster_algorithm_2010} with the following additional changes: The hard iteration limits MaxSweeps and MaxIter are removed for simplicity. Although these are not necessary, one may still wish to include them in an applied implementation of the algorithm to safeguard against excessively small choices of $\epsilon$. $\beta$ and $\|\bullet\|$ are introduced to allow for generalisation. In \citet[Table I]{foster_algorithm_2010} $\beta(p(z))=\frac{p_\ell}{\|p_\ell\|_2} z^\ell$ where in the polynomial $p(z)$, the monomial $p_\ell z^\ell$ has the largest coefficient (in absolute value), and $\|\bullet\|=\max\limits_{j \in \setZ} \|p_j\|_2$.
Our definition of a Givens rotation includes pre- and post-rotation shifting whereas they consider post-rotation shifting as a separate operation. Finally, we added lines \ref{line:extrastart}--\ref{line:extraend} to improve stability by making the real part of the diagonal entries as large as easily achievable, and lines \ref{line:extrastart2}--\ref{line:extraend2} to make the real part of the diagonal entries positive (which in certain cases ensures uniqueness).

The output from Algorithm~\ref{code:QR} satisfies $\bA=\bQ\bR$, and $\bQ$ is unitary. However, $\bR$ may be only approximately upper triangular, in the sense that every entry below the diagonal has norm at most $\epsilon$. In most applications this will be acceptable as long as a sufficiently small $\epsilon$ is chosen.

\subsection{Choosing $\beta$}

\begin{definition}\label{def:decent}
Given a norm $\|\bullet\|$, a function $\beta:\setA \rightarrow \mathcal{U}(\setA)$ is said to be \emph{decent} (or $\|\bullet\|$-decent, or $(\|\bullet\|,\rho)$-decent)  if there exists $\rho>0$ such that
\begin{align}
&\left|\Re\left(\overline{\beta(a)}a\right)\right| \geq \rho\|a\|\ \ \ \forall a \in \setA, \label{eq:betacond1}\\
&\left|\Re\left(\overline{\beta(a)}a\right)\right| \geq \left|\Re(a)\right|  \ \ \ \forall a \in \setA.\label{eq:betacond2}
\end{align}
\end{definition}
\begin{theorem}\label{thm:convergence}
If $\beta$ is $\|\bullet\|$-decent
then Algorithm~\ref{code:QR} converges.
\end{theorem}
\begin{proof}
Note that A2 implies that $\setA$-Givens rotations are isometries, so that $\|\bR\|_F^2=\|\bA\|_F^2$ is constant throughout Algorithm~\ref{code:QR}.

First we fix $k$ and prove by contradiction that the inner loop (lines~\ref{line:innerloopstart}--\ref{line:innerloopend}) converges, i.e.\ $g_2 \leq \epsilon$ eventually.
By Lemma~\ref{lm:givens}, each Givens rotation (line~\ref{line:givens}) increases $\Re(r_{kk})^2$ by $\Re(\overline{\beta(r_{ik})}r_{ik})^2
$. Because the monotonically increasing sequence of $\Re(r_{kk})^2$ is bounded above by $\|\bR\|_F^2$, it must converge by the monotone convergence theorem towards some quantity $s \geq 0$. Hence there is some point after which $|\Re(r_{kk})^2 -s|< \rho^2\epsilon^2$. At the next iteration $\Re(r_{kk})^2$ increases by less than $\rho^2\epsilon^2$, which implies $\rho^2\|r_{ik}\|^2 \leq \Re(\overline{\beta(r_{ik})}r_{ik})^2 < \rho^2\epsilon^2$ using (\ref{eq:betacond1}). Hence $g_2=\|r_{ik}\| < \epsilon$.

Now we will prove by contradiction that the outer loop (lines \ref{line:outerloopstart}--\ref{line:outerloopend}) converges, i.e.\ $g_1 \leq \epsilon$ eventually. 
Consider $\Re(r_{jj})^2$. Line~\ref{line:extramid} cannot decrease $\Re(r_{jj})^2$ by (\ref{eq:betacond2}). The inner loop will not affect $\Re(r_{jj})^2$ when $k>j$ and we have already shown above that the inner loop cannot decrease $\Re(r_{jj})^2$ when $k=j$. Hence $\Re(r_{jj})^2$ may only decrease when $k <j$. Hence $\Re(r_{11})^2$ is monotonically increasing. Hence (using the same argument as for the inner loop) there is some $s_1 \geq 0$ such that there is some point after which  $|\Re(r_{11})^2 -s_1|< \rho^2\epsilon^2$ and from then onwards $\max\limits_{\ell:\ell>1}\|r_{\ell 1}\| <\epsilon$. From then onwards the inner loop will have no effect on $\bR$ when $k=1$.
Now proceed by strong induction on $j$. If $\max\limits_{\ell,j^\prime:\ell> j^\prime,j^\prime \leq j}\|r_{\ell j^\prime}\| <\epsilon$ then the inner loop will not affect $\bR$ when $k \leq j$. Hence $\Re(r_{j+1,j+1})^2$ is monotonically increasing and (using the same argument as for $\Re(r_{11})^2$) we will reach a point after which $\max\limits_{\ell:\ell> j+1}\|r_{\ell,j+1}\| <\epsilon$. Hence by strong induction we will reach a point where $g_1 = \max\limits_{\ell,j^\prime:\ell>j^\prime }\|r_{\ell j^\prime}\| < \epsilon$.
\end{proof}
\begin{remark}\label{rk:betaprime}
If $\beta$ satisfies (\ref{eq:betacond1}) but not (\ref{eq:betacond2}), then
\[
\beta^\prime(a)=\begin{cases}
\beta(a) &\mbox{ if } \left|\Re\left(\overline{\beta(a)}a\right)\right| \geq \left|\Re(a)\right|\\
1 & \mbox{ otherwise}
\end{cases}
\]
satisfies both (\ref{eq:betacond1}) and  (\ref{eq:betacond2}).
\end{remark}
We will without loss of generality define $\beta(0)=1$ from now on and assume that $a \neq 0$ when computing or defining $\beta(a)$.
\begin{lemma}\label{lm:comparedecency}
If $\beta$ is $(\|\bullet\|,\rho)$-decent, then it is also $(\|\bullet\|,\rho^\prime)$-decent $\forall \rho^\prime\leq\rho$.\\
If $\beta$ is $(\|\bullet\|_2,\rho)$-decent, then it is also $(\|\bullet\|_\infty,\rho)$-decent.\\
If $\beta$ is $(\|\bullet\|_\infty,\rho)$-decent, then it is also $(\|\bullet\|_2,\rho d^{-\frac{1}{2}})$-decent.
\end{lemma}
\begin{proof}
(\ref{eq:betacond2}) does not depend on $\|\bullet\|$ or $\rho$. In (\ref{eq:betacond1}) we have $\rho\|a\| \geq \rho^\prime \|a\|$ and $\rho\|a\|_2 \geq \rho\|a\|_\infty \geq \rho d^{-\frac{1}{2}} \|a\|_2$.
\end{proof}

The question remains of choosing an appropriate function $\beta$ (and norm $\|\bullet\|$) satisfying the assumptions of Definition~\ref{def:decent}. In general the larger $\left|\Re\left(\overline{\beta(a)}a\right)\right|$ is, the faster the algorithm will converge, so we want $\left|\Re\left(\overline{\beta(a)}a\right)\right|$ (and hence $\rho$) to be large if possible. Hence the obvious choice is 
\begin{equation}\label{eq:betamax}
\beta_{\max}(a)=\argmax\limits_{b \in \mathcal{U}(\setA)} \Re\left(\bar{b}a\right).
\end{equation}
One additional requirement is that we should be able to compute $\beta(a)$ in a reasonable amount of time, so that a simpler choice may sometimes be preferable.

Let $\|\bullet\|=\|\bullet\|_2$. Then $\|a\|_2^2 = \left\|\overline{\beta(a)}a\right\|_2^2 \geq \Re\left(\overline{\beta(a)}a\right)^2$, and the best we can hope for is $\rho=1$, which happens when $\overline{\beta(a)}a =\Re\left(\overline{\beta(a)}a\right) \in \setR$, 
which implies $\left(\overline{\beta(a)}a\right)^{-1}\overline{\beta(a)}=a^{-1}$ (or $a = 0$). Hence $\rho=1$ is only possible when every non-zero $a \in \setA$ has an inverse, i.e.\ when $\setA$ is a division algebra. But Frobenius' theorem states that the only finite-dimensional real division algebras (up to isomorphism) are $\setR$, $\setC$ and $\setH$ \citep{palais_classification_1968}. For $\setA=\setR$ we can choose $\beta(a)=1$. For $\setA=\setC$ or $\setA=\setH$ we can choose $\beta(a)=\beta_{\max}(a)=\frac{a}{\|a\|_2}$,
(assuming $\bar{\bullet}$ is the usual complex or quaternion  conjugation). In those three cases Algorithm~\ref{code:QR} will be a standard real/complex/quaternion QR by columns algorithm, and will converge in a finite number of steps even when $\epsilon=0$, so that $\bR$ is exactly triangular. In addition, the diagonal entries of $\bR$ will be real and positive in these three cases.\label{rk:divisionalgebra}

\begin{example}
Let $\setA=\setR^{d^\prime \times d^\prime}$ ($d={{}d^\prime}^2$), then $\beta_{\max}$ can be computed through $\beta_{\max}(a)=uv^\tp$, where $u,v \in \mathcal{U}(\setA)$ are obtained from the real SVD $a=udv^\tp$ \citep{fan_metric_1955}. Similarly, for $\setA=\setC^{d^\prime \times d^\prime}$ ($d=2{{}d^\prime}^2$) or $\setA=\setH^{d^\prime \times d^\prime}$ ($d=4{{}d^\prime}^2$), we have $\beta_{\max}(a)=uv^\herm$ where $u,v \in \mathcal{U}(\setA)$ are obtained from the complex/quaternion SVD $a=udv^\herm$. In all three cases $\beta_{\max}$ is $(\|\bullet\|_F,{{}d^\prime}^{-1})$-decent (or equivalently $(\|\bullet\|_2,{{}d^\prime}^{-\frac{1}{2}})$-decent).
\end{example}

We will now consider constructions for $\beta$ which are appropriate when $\setA$ is not a division algebra, and when computation of $\beta_{\max}$ is difficult or slow.

\begin{definition}\label{def:betamax}
Let $\mathcal{B}=\{e_1,e_2,\ldots,e_d\}$ be a basis of $\setA$,
 and $a=a_1 e_1 + \ldots+ a_d e_d \in \setA$. Define $J(a)= \argmax\limits_j |a_j|$ and $\beta_\mathcal{B}(a)=e_{J(a)} \in \mathcal{B}$.
\end{definition}
\begin{definition}\label{def:unitary}
The basis  $\mathcal{B}=\{e_1,e_2,\ldots,e_d\}$ of $\setA$ is \emph{unitary} if $\mathcal{B} \subseteq \mathcal{U}(\setA)$ and $\Re(\bar{e_i}e_j)=\delta_{ij} \ \forall i,j$. 
\end{definition}
Unitary bases for the case $\setA= \setC^{d^\prime \times d^\prime}$ (where the field is $\setK=\setC$ and $d^\prime=\sqrt{d}$), are of interest in the quantum mechanics, quantum computing and quantum error correcting code literature, where they are called unitary error bases. Methods based on latin squares and projective representations for constructing such a basis explicitly for arbitrary $d^\prime$ are explained in \citet{klappenecker_unitary_2003}.
Note that if $\{e_1,\ldots,e_{{{}d^\prime}^2}\}$ is such a unitary error basis for $\setC^{d^\prime \times d^\prime}$ with $\setK=\setC$ (i.e.\ viewed as a complex algebra), then $\{e_1,\rmi e_1,\ldots,e_{{{}d^\prime}^2},\rmi e_{{{}d^\prime}^2}\}$ is a unitary basis for $\setC^{d^\prime \times d^\prime}$ with $\setK=\setR$ (i.e.\ viewed as a real algebra with $d=2{{}d^\prime}^2$), and we also have that $\{e_1,\rmi e_1,\rmj e_1,\rmk e_1,\ldots,e_{{{}d^\prime}^2},\rmi e_{{{}d^\prime}^2},\rmj e_{{{}d^\prime}^2},\rmk e_{{{}d^\prime}^2}\}$ is a unitary basis for $\setH^{d^\prime \times d^\prime}$ (with $\setK=\setR$ and $d=4{{}d^\prime}^2$).

One can show that no unitary basis exists for the simple algebra $\setR^{3 \times 3}$ with $\Re=\frac{1}{3}\trace$,%
\footnote{Indeed there do not exist four mutually orthogonal $3 \times 3$ orthogonal matrices, let alone nine.}
and this implies that although all finite-dimensional semi-simple algebras have an invertible basis \citep[Corollary~3.2.7]{lopez-permouth_algebras_2015}, they do not all have a unitary basis. 

The existence of a unitary basis for $\setR^d$ (with element-wise multiplication) is equivalent to the existence of a (real) $d \times d$ Hadamard matrix. It is known that (real) $d \times d$ Hadamard matrices do not exist when $d \geq 3$, $d \neq 0\! \mod 4$. The existence of (real) Hadamard matrices for  all $d = 0\! \mod 4$ is a long-standing open conjecture in coding theory \citep{weisstein_hadamard_2015}.

\begin{lemma}\label{lm:unitary}
If $\mathcal{B}=\{e_1,e_2,\ldots,e_d\}$ is a unitary basis of $\setA$, then $\beta_\mathcal{B}$ in Definition~\ref{def:betamax} is $\left(\|\bullet\|_\infty,1\right)$-decent (and $\left(\|\bullet\|_2,d^{-\frac{1}{2}}\right)$-decent).
\end{lemma}
\begin{proof} 
\[
\Re\left(\overline{\beta(a)} a\right)^2=a_{J(a)}^2=\|a\|_{\infty}^2 \geq \frac{1}{d} \|a\|_2^2\ ,\ \ \mbox{ and }\ \ \|a\|_{\infty}^2 \geq \Re(a)^2.
\]
\end{proof}

\begin{definition}
The real twisted group algebra $\setR^\alpha[G]$ obtained from a group $G$ with twisting function $\alpha: G \times G \rightarrow \{-1,1\}$ is the real algebra $\setR[\mathcal{B}]$ generated by the basis $\mathcal{B}=\{b_g:g \in G\}$ with multiplication $b_g b_h=\alpha(g,h) b_{gh}$. To preserve associativity we require that $\alpha$ satisfies
\begin{equation}\label{eq:associativity}
\alpha(f,g)\alpha(fg,h)=\alpha(f,gh)\alpha(g,h)\ \forall f,g,h \in G.
\end{equation}
If $\alpha(g,h)=1 \ \forall g,h \in G$ then $\setR^\alpha[G]$ is a group algebra denoted $\setR[G]$.
\end{definition}
The twisted group algebras $\setR^\alpha[G]$ and $\setR^{-\alpha}[G]$ are isomorphic, hence we will from now on assume without loss of generality that $\alpha(1,1)=1$.


\begin{proposition}\label{prop:twistedalgebra}
Let $G$ be a (finite) group and let $\setA$ be the twisted group algebra $\setR^\alpha[ G]=\setR[ \mathcal{B}]$, with basis $\mathcal{B}=\left\{b_g : g \in G\right\}$.   Let $\bar{\bullet}$ be the unique involution such that $\bar{b}=b^{-1}$ for all $b \in \mathcal{B}$.
Then A1, A2, A3 are satisfied, and $\mathcal{B}$ is a unitary basis.
\end{proposition}
\begin{proof}
Setting $g=1$ in (\ref{eq:associativity}) implies $\alpha(f,1)=\alpha(1,h)=\alpha(1,1) \ \forall f,h\in G$, which using the assumption $\alpha(1,1)=1$ implies $b_1=1$ and A1 is satisfied.

The matrices $\tilde{b} \in \tilde{\mathcal{B}}$ are of the form $\tilde{b}=\bD\bP$ where $\bD$ is diagonal with $\pm 1$ on its diagonal entries, and $\bP$ is a permutation matrix, hence $\tilde{\mathcal{B}}\subseteq \mathcal{U}(\setR^{d \times d})$ so that $\tilde{\bar{b}}=\widetilde{b^{-1}}=\tilde{b}^{-1}=\tilde{b}^\tp$. In particular $\tilde{\setA}$ is closed under $\bullet^\tp$ since $\tilde{b_g}^\tp=(\tilde{b_g})^{-1}=\alpha(g,g^{-1})\widetilde{b_{g^-1}} \in \tilde{\setA}$. Hence A3 (and A2) is satisfied.

The above is also a direct consequence of \citet[Theorems~4.11, 5.2, 5.3 and 5.5]{bales_properly_2006}.

\item $\bar{b_g}b_h=b_g^{-1} b_h=\alpha(g,g^{-1})b_{g^{-1}}b_h=\alpha(g,g^{-1})\alpha(g^{-1},h)b_{g^{-1}h}$. Hence either the real part is $0$, or  $g^{-1}h=1$. In the latter case $h=g$ and we have $\Re(\bar{b_g}b_g)=\Re(1)=1$.
\end{proof}


%

\section{The $\setA$SVD by $\setA$QRD algorithm}\label{sec:SVD}

In this section we show that any convergent $\setA$QRD algorithm also yields a convergent $\setA$SVD algorithm. This is essentially the same as the approach taken in \citet[Table~II]{foster_algorithm_2010} for the special case where $\setA=\setC[z,z^{-1}]$ is the infinite-dimensional algebra of Laurent polynomials with complex coefficients.

\begin{algorithm}[htbp]
\caption{$\setA$SVD by $\setA$QRD}\label{code:SVD}
\begin{algorithmic}[1]
\State \textbf{Input:} The matrix to be decomposed $\bA \in \setA^{m \times n}$
\State \textbf{Specify:} A function $\beta:\setA \rightarrow \mathcal{U}(\setA)$,\\ \hspace{4.5em}a norm $\|\bullet\|$,\\ \hspace{4.5em}and the error tolerance $\epsilon>0$.
\State\textbf{Initialise:} $\bU \gets \bI_m$, $\bV \gets \bI_n$, $\bD \gets \bA$. 
\State $g \gets \max\limits_{i,j: i\neq j} \|d_{ij}\|$
\While{$g>\epsilon$} 
\State Compute $\bR$ and $\bQ$ from the $\setA$QRD of $\bD$ using Algorithm~\ref{code:QR}
\State $\bD \gets \bR$
\State $\bU \gets \bU \bQ$
\State Compute $\bR$ and $\bQ$ from the $\setA$QRD of $\bD^\herm$ using Algorithm~\ref{code:QR}
\State $\bD \gets \bR^\herm$
\State $\bV \gets \bV \bQ$
\State $g \gets \max\limits_{i,j: i\neq j} \|d_{ij}\|$
\EndWhile 
\State \textbf{Output:} $\bU$, $\bD$ and $\bV$
\end{algorithmic}
\end{algorithm}

The output from Algorithm~\ref{code:SVD} satisfies $\bA=\bU\bD\bV^\herm$, and $\bU$, $\bV$ are both unitary. However, $\bD$ may be only approximately diagonal, in the sense that every off-diagonal entry has norm at most $\epsilon$. In most applications this will be acceptable as long as a sufficiently small $\epsilon$ is chosen.

\begin{theorem}\label{thm:convergence2}
If $\beta$ is $\|\bullet\|$-decent
then Algorithm~\ref{code:SVD} converges.
\end{theorem}
\begin{proof}
Theorem~\ref{thm:convergence} states that each QRD step converges. It remains to show that eventually $g\leq \epsilon$.
The proof will proceed by contradiction similarly to the proof of Theorem~\ref{thm:convergence}. 
$\Re(d_{11})^2$ is increasing and bounded above by $\|\bA\|_F^2$. Hence there is some point after which $|\Re(d_{11})^2 -s_1|< \rho^2\epsilon^2$, so that none of the Givens rotations in either QR step may then increase $\Re(d_{11})^2$ by more than $\rho^2\epsilon^2$, so $\max\limits_{i>1}\max(\|d_{i1}\|,\|d_{1i}\|) < \epsilon$.
From this point onwards the first row and column will remain fixed. If $\max\limits_{i,k:k \leq j, k < i}\max(\|d_{ik}\|,\|d_{ki}\|) < \epsilon$ then the first $j$ rows and columns are fixed and $\Re(d_{j+1,j+1})^2$ is monotonically increasing. Using the same argument on $\Re(d_{j+1,j+1})^2$ as on $\Re(d_{11})^2$ there is some point after which $\max\limits_{i>j+1}\max(\|d_{i,j+1}\|,\|d_{j+1,i}\|) < \epsilon$. Proceeding by strong induction on $j$ we eventually have $g=\max\limits_{i,k: k < i}\max(\|d_{ik}\|,\|d_{ki}\|) < \epsilon$.
\end{proof}

\citet{albuquerque_clifford_2002} show that every Clifford algebra $C\ell(p,q)$ can be constructed as a twisted group algebra with group $G=\left(\frac{\setZ}{2\setZ}\right)^{p+q}$ (\citet[Section~8]{bales_properly_2006} shows this for the case $p=0$). Hence we may apply our results to Clifford algebras.
The reduced quaternions and quad-quaternions are also examples of twisted group algebras.

\begin{corollary}
Let $\setA$ be the $2^{p+q}$-dimensional real Clifford algebra $C\ell(p,q)$ with standard basis of blades $\mathcal{B}$. Then A1, A2, A3 are satisfied, $\mathcal{B}$ is a unitary basis, $\beta_\mathcal{B}(a)=e_{\argmax\limits_j|a_j|}$ is $\left(\|\bullet\|_\infty,1\right)$-decent, and Algorithms~\ref{code:QR}\,\&\,\ref{code:SVD} converge taking $\|\bullet\|=\|\bullet\|_\infty$, $\beta=\beta_\mathcal{B}$.
\end{corollary}
\begin{proof}
This follows immediately from Proposition~\ref{prop:twistedalgebra}, Lemma~\ref{lm:unitary}, Theorems~\ref{thm:convergence}\,\&\,\ref{thm:convergence2} and \citep{albuquerque_clifford_2002}.
\end{proof}

 It is important at this point to bear in mind that Proposition~\ref{prop:twistedalgebra} (or equivalently assumption A3) \emph{defines} the involution on $\setA$. 
 Different choices of involution would imply different notions of orthogonality and a different definition of unitary matrix. For the Clifford algebra $C\ell(p,q)$ the involution imposed by Proposition~\ref{prop:twistedalgebra} will be the ``Hermitian conjugation'' of 
 \citet{marchuk_unitary_2008}, which is in general different from the standard ``Clifford involution''. 
The reduced quaternion SVD of \citet{gai_denoising_2015,gai_reduced_2014} uses the involution obtained through A3 although this is not stated explicitly.
One may very well wish to specify the involution, so this forced choice may seem restrictive. It is however not possible in general to compute an SVD with an arbitrary choice of involution. For example, \citet{verstraete_lorentz_2002} use the indefinite Minkowski inner product to define orthogonality, and note that the (real) SVD can no longer diagonalise all matrices when orthogonal matrices are replaced with finite Lorentz transformations.\footnote{This conclusion would remain even without restricting themselves to proper orthochronous Lorentz transformations.} 
Even in the simple case $\setA=\setC$ with the identity as its involution instead of complex conjugation, we have $\mathcal{U}(\setC)=\{-1,1\}$ and no decent $\beta$ exists.

\begin{table}
\centering
\begin{tabular}{|c|c|c|l|}
\hline
Algebra & $\|a\|$ & $\beta(a)$ & Notes\\
\hline\hline
$\setR$ & $|a|$ & 1 & $\beta=\beta_\mathcal{B}=(\pm)\beta_{\max}$\\
$\setC$ & $\|a\|_2$ & $\frac{a}{\|a\|_2}$ & $\beta=\beta_{\max}$\\
$\setH$ & $\|a\|_2$ & $\frac{a}{\|a\|_2}$ & $\beta=\beta_{\max}$\\
$\setR[z,z^{-1}]$ & $\|a\|_\infty$ & $z^{J(a)}$ & $\beta=\beta_\mathcal{B}=(\pm)\beta_{\max}$, $a=\sum_{j\in \setZ} a_j z^j$\\
 $\setC[z,z^{-1}]$ & $\max\limits_{j \in \setZ} \|a_j\|_2$ & $\frac{a_{J(a)}}{\|a_{J(a)}\|_2}z^{J(a)}$ & $\beta=\beta_{\max}$, $a=\sum_{j\in \setZ} a_j z^j$\\
 $\setH[z,z^{-1}]$ & $\max\limits_{j \in \setZ} \|a_j\|_2$ & $\frac{a_{J(a)}}{\|a_{J(a)}\|_2}z^{J(a)}$ & $\beta=\beta_{\max}$, $a=\sum_{j\in \setZ} a_j z^j$\\
$C\ell(p,q)$ & $\|a\|_\infty$ & $e_{J(a)}$ &  $\beta=\beta_\mathcal{B}$, $\bar{e_j}=e_j^{-1}$\\
\hline
\end{tabular}
\caption{
Some standard and recommended choices of norm $\|\bullet\|$ and function $\beta$ to use in Algorithms~\ref{code:QR}\,\&\,\ref{code:SVD}. All of the $\beta$ given are $1$-decent. Here $J(a)=\argmax_j \|a_j\|_2$. Note that $\setR$, $\setC$ and $\setH$ are the only algebras for which one can set $\epsilon=0$ in Algorithm~\ref{code:QR}.
}
\end{table}

\section{Using the classification of real semi-simple algebras}\label{sec:semisimple}
When A3 holds we may assume without loss of generality that $\setA$ is an algebra of real matrices. It is then helpful to think of abstract tensor products $\otimes$ and direct sums $\oplus$ in terms of explicit Kronecker products and Kronecker sums.
\begin{definition}
Let $\setA$ be a $d$-dimensional real vector space with basis $\mathcal{B}=\{e_1,\ldots,e_d\}$ and let $\setA^\prime$ be a $d^\prime$-dimensional real vector space with basis $\mathcal{B}^\prime=\{e_1^\prime,\ldots,e_{d^\prime}^\prime\}$. Then
\begin{itemize}
\item $\setA \oplus \setA^\prime$ is a $(d+d^\prime)$-dimensional real vector space with basis 
$\mathcal{B}^\oplus=\{e_1 \oplus 0,\ldots, e_d \oplus 0, 0 \oplus e_1^\prime,\ldots,0 \oplus e_{d^\prime}^\prime \}$.%
\footnote{The particular ordering used here assumes that $d<\infty$.}
\item $\setA \otimes \setA^\prime$ is a $(d\cdot d^\prime)$-dimensional vector space with basis
$\mathcal{B}^\otimes=\{e_1\otimes e_1^\prime,\ldots,e_1\otimes e_{d^\prime}^\prime,\ldots,e_d \otimes e_1^\prime , \ldots, e_d \otimes e_{d^\prime}^\prime\}$.%
\footnote{The particular ordering used here assumes that $d^\prime<\infty$.}
\end{itemize}
If both $\setA$ and $\setA^\prime$ are real algebras, then furthermore
\begin{itemize}
\item $\setA \oplus \setA^\prime$ is a real algebra with multiplication $(x \oplus x^\prime)(y \oplus y^\prime)=(x x^\prime) \oplus (y y^\prime)$.
\item  $\setA \otimes \setA^\prime$ is a real algebra with multiplication satisfying $(x \otimes x^\prime)(y \otimes y^\prime)=(x x^\prime) \otimes (y y^\prime)$.
\end{itemize}
If both $\setA$ and $\setA^\prime$ are $*$-algebras then furthermore
\begin{itemize}
\item $\setA \oplus \setA^\prime$ is a $*$-algebra with involution $\overline{a \oplus a^\prime}=\bar{a} \oplus \bar{a^\prime}$.
\item $\setA \otimes \setA^\prime$ is a $*$-algebra with involution satisfying $\overline{a \otimes a^\prime}=\bar{a} \otimes \bar{a^\prime}$.
\end{itemize}
\end{definition}
\begin{proposition}
Let $\setA$ be a $d$-dimensional real algebra with unitary basis $\mathcal{B}=\{e_1,\ldots,e_d\}$ and let $\setA^\prime$ be a $d^\prime$-dimensional real algebra with unitary basis $\mathcal{B}^\prime=\{e_1^\prime,\ldots,e_{d^\prime}^\prime\}$.
Then $\mathcal{B}^\otimes$ is a unitary basis for $\setA \otimes \setA^\prime$.

If furthermore both $\setA$ and $\setA^\prime$ satisfy A1 (resp. A3) then  $\setA \otimes \setA^\prime$ satisfies A1 (resp. A3).

Also, if $d=d^\prime$ then $\mathcal{B}^\pm=\{e_1\oplus e_1^\prime,e_1 \oplus -e_1^\prime,\ldots,e_d \oplus e_d^\prime ,e_d \oplus -e_d^\prime\}$ is a unitary basis for $\setA \oplus \setA^\prime$.
If furthermore both $\setA$ and $\setA^\prime$ satisfy A1 (resp. A3) then (using the basis $\mathcal{B}^\pm$)  $\setA \oplus \setA^\prime$ satisfies A1 (resp. A3).

\end{proposition}
\begin{proof}

The proofs involve checking of the definitions, which is straightforward once one notes that $\mathcal{U}(\setA \oplus \setA^\prime)=\mathcal{U}(\setA) \oplus \mathcal{U}(\setA^\prime)$, $\mathcal{U}(\setA \otimes \setA^\prime)\supseteq \mathcal{U}(\setA) \otimes \mathcal{U}(\setA^\prime)$, $\Re(a \oplus a^\prime)=\frac{1}{2}(\Re(a)+\Re(a^\prime))$ (when $d=d^\prime$) and $\Re(a \otimes a^\prime)=\Re(a)\Re(a^\prime)$.
\end{proof}

Consider the case $\setA = \setR^{k \times k}$, $\bar{\bullet}=\bullet^\tp$. We can identify $\bA \in \setA^{m \times n}=\setR^{m \times n} \otimes \setA$ with an $m \times n$ block-matrix with $k \times k$ blocks $\bM \in \setR^{mk \times nk}$ (the $i^\prime,j^\prime$-entry of the $i,j$-entry of $\bA$ is the $ki+i^\prime,kj+j^\prime$-entry of $\bM$). Similarly for $\setA^{m \times m}$, $\setA^{n \times n}$. Every upper triangular (resp. diagonal) matrix in $\setR^{mk \times mk}$ is also upper triangular (resp. diagonal) in $\setA^{m \times n}$.
 The unitary matrices in $\setA^{m \times m}$ (resp. $\setA^{n \times n}$) are orthogonal matrices in $\setR^{mk \times mk}$ (resp. $\setR^{nk \times nk}$) and vice-versa.
 A $\setR$QRD (resp. $\setR$SVD) of $\bM \in \setR^{mk \times nk}$ is thus immediately also an $\setA$QRD (resp. $\setA$SVD) of $\bA \in \setA^{m \times n}$. This block-matrix based approach remains valid if we replace $\setR$ with $\setC$ or $\setH$ and $\bullet^\tp$ with $\bullet^\herm$. In other words, the QRD (resp. SVD) of a (block-)matrix is also a valid block-QRD (resp. block-SVD).

\citet{gai_reduced_2014} uses the fact that the reduced quaternion algebra is isomorphic to $\setC \oplus \setC$ to compute the reduced quaternion SVD through two parallel $\setC$SVDs. We will now generalise this technique to all semi-simple algebras.

Let $\setA=\bigoplus_{\ell=1}^s\setA_\ell$.  We identify $\setA_k$ with the subalgebra $\left(\bigoplus_{\ell=1}^{k-1} \{0\}\right) \oplus \setA_k \oplus \left(\bigoplus_{\ell=k+1}^{s} \{0\}\right)$ of $\setA$. Let $1_k= \left(\bigoplus_{\ell=1}^{k-1} 0\right) \oplus 1 \oplus \left(\bigoplus_{\ell=k+1}^{s} 0\right) \in \setA$ denote the identity of $\setA_k$.  Since
$1_k^2=1_k=\bar{1_k}$ 
 and 
$1_k 1_\ell =0  \;\forall\, k \neq \ell$,
 multiplication by $1_k$ is an orthogonal projection into the subalgebra $\setA_k$. Since $ \setA \ni 1= \sum_{\ell=1}^s 1_\ell$, every element $a \in \setA$ can be written as $\sum_{\ell=1}^s a 1_\ell$, where $a 1_k\in \setA_k$.

Let $\bA \in \setA^{m \times n}$ and $\bA_k = \bA 1_k$. We may treat $\bA_k$ as belonging to $\setA_k^{m \times n}$ and compute its  $\setA_k$QRD $\bA_k = \bQ_k \bR_k$. These $s$ $\setA_k$QRDs can be added to form the $\setA$QRD $\bA= \left(\sum_{\ell=1}^s \bQ_\ell 1_\ell \right)\left(\sum_{\ell=1}^s \bR_\ell 1_\ell \right)$. Similarly, the $s$ $\setA_k$SVDs  $\bA_k= \bU_k \bD_k \bV_k^\herm$  can be added to form the $\setA$SVD \linebreak$\bA= \left(\sum_{\ell=1}^s \bU_\ell 1_\ell \right) \left(\sum_{\ell=1}^s \bD_\ell 1_\ell \right) \left(\sum_{\ell=1}^s \bV_\ell 1_\ell \right)^\herm$.

Because of Frobenius' theorem \citep{palais_classification_1968}, in the case of real algebras the Artin-Wedderburn theorem \citep{grillet_abstract_2007} can be stated as
\begin{proposition}\label{cor:artinwedderburn}
Every finite-dimensional real semi-simple algebra is isomorphic to a direct sum
\begin{equation*}
\setA_1^{n_1 \times n_1} \oplus \cdots \oplus \setA_s^{n_s \times n_s}
\end{equation*}
of finitely many matrix algebras where $\setA_1,\ldots., \setA_s \in \lbrace \setR, \setC, \setH \rbrace$.
\end{proposition}
\begin{corollary}\label{cor:semisimple}
If $\setA$ is a semi-simple algebra, then after choosing the representation of $\setA$ given by Proposition~\ref{cor:artinwedderburn}, with corresponding quaternion matrix involution $\bullet^\herm$,  an $\setA$QRD (resp. $\setA$SVD) of $\bA \in \setA^{m \times n}$ can be obtained by  computing $s$ independent $\setA_k$QRDs (resp. $\setA_k$SVDs) of matrices $\bA_k \in \setA_k^{n_k m \times n_k n}$, where $\setA_1,\ldots., \setA_s \in \lbrace \setR, \setC, \setH \rbrace$.
\end{corollary}
Thus real, complex and quaternion QRD (resp. SVD) algorithms are in principle sufficient to compute $\setA$QRDs (resp. $\setA$SVDs) for any semi-simple algebra $\setA$. One caveat is that to use this result in practice one must be able to compute the isomorphism between $\setA$ and $\bigoplus_{\ell=1}^s \setA_\ell^{n_\ell \times n_\ell}$ explicitly. Because real algebras are also real vector spaces, the isomorphism will be an invertible linear change of basis.

Every Clifford algebra $\mathcal{C}\ell(p,q)$ is either a simple algebra isomorphic to $\setR^{n \times n}$, $\setC^{n \times n}$, or $\setH^{n \times n}$  or a semi-simple algebra isomorphic to $\setR^{n \times n} \oplus \setR^{n \times n}$
, or $\setH^{n \times n} \oplus \setH^{n \times n}$. Explicit constructions of these isomorphisms are available in \citep{tian_universal_1998}.

The use of Corollary~\ref{cor:semisimple} to compute the QRD has advantages over the direct use of Algorithm~\ref{code:QR}. It is inherrently parallelised, and as mentioned in page~\pageref{rk:divisionalgebra}, we may set $\epsilon=0$ when computing QRDs in $\setR$,$\setC$ or $\setH$. We would also expect it to be more computationally efficient, epecially for small $\epsilon$. This expectation is confirmed in Section~\ref{sec:CGA} Figure~\ref{fig:rotations}.

\section{Examples}\label{sec:examples}

%

\subsection{Multivariate Laurent polynomial algebra}
Let $z_1,\ldots,z_\kappa$ be $\kappa$ abstract commuting variables . The infinite-dimensional algebra of $\kappa$-variate Laurent polynomials with real coefficients is \linebreak$\setA=\setR[z_1,z_1^{-1},\ldots,z_\kappa,z_\kappa^{-1}]$. It is a commutative algebra and its standard basis is the set of monomials $\mathcal{B}=\{\prod_{k=1}^\kappa z_k^{\rho_k}:\rho_1,\ldots,\rho_\kappa \in \setZ\}$. $\mathcal{B}$ is a multiplicative group which is isomorphic to the additive group $\setZ^\kappa$, so $\setA$ is a group algebra. Each element of $\setA$ is by definition a \emph{finite} linear combination of monomials.

If $\kappa=1$ then $\setA$ is the usual Laurent polynomials, and can be used to represent time-series and convolutive filters acting on them, as in \citet{mcwhirter_evd_2007}. Setting $\kappa=2$ $\setA$ allows the same type of analysis to be performed on images and convolutive filters acting on them (e.g. blurring and 2D shifting), $\kappa=3$ can be used for 3D images or videos, and so forth.

By Proposition~\ref{prop:twistedalgebra}, Lemma~\ref{lm:unitary} and Theorem~\ref{thm:convergence} (resp. Theorem~\ref{thm:convergence2})  the QRD (resp. SVD) of a multivariate Laurent polynomial matrix can be computed using Algorithm~\ref{code:QR} (resp. Algorithm~\ref{code:SVD}) with $\overline{\prod_{k=1}^\kappa z_k^{\rho_k}}=\prod_{k=1}^\kappa z_k^{-\rho_k}$, $\|\bullet\|=\|\bullet\|_\infty$ and $\beta=\beta_\mathcal{B}$ (see Definition~\ref{def:betamax}).

Because $\setA$ is infinite-dimensional, the approach of Section~\ref{sec:semisimple} cannot be applied directly. However, the group $\setZ$ can be approximated for large $\delta$ by the finite cyclic group $\mathcal{C}_\delta=\frac{\setZ}{\delta\setZ}$, and by extension the algebra $\setA \cong \setR[\setZ^\kappa]$ can be approximated by $\setA^\prime=\setR[\mathcal{C}_\delta^\kappa]$. 
More precisely, any finite sequence of calculations performed in $\setR[\setZ^\kappa]$ and in $\setR[\mathcal{C}_\delta^\kappa]$ will produce identical results if $\delta$ is sufficiently large. If $\delta$ is larger than twice the largest (positive or negative) power of $z_k$ in a signal, then viewing that signal as belonging to $\setR[\mathcal{C}_\delta^\kappa]$ instead of $\setR[\setZ^\kappa]$ is conceptually the same as using the periodic edge extension convention on the signal of size $\delta \times \cdots \times \delta$ rather than the zero-padding edge extension convention.

We can now use the approach of Section~\ref{sec:semisimple} to compute the QRD or SVD of matrices in $\setR[\mathcal{C}_\delta^\kappa]$ by noting that for even $\delta$ $\setR[\mathcal{C}_\delta^\kappa]$ is ($*$-)algebra-isomorphic to $\setC^{\left(\frac{\delta}{2}\right)^\kappa}=\bigoplus_{\ell=1}^{\left(\frac{\delta}{2}\right)^\kappa} \setC$, with the isomorphism given explicitly by the positive frequencies of the $\kappa$-dimensional discrete Fourier transform. This ``approximate isomorphism'' between $\setR[z_1,z_1^{-1},\ldots,z_\kappa,z_\kappa^{-1}]$ and $\setC^{\left(\frac{\delta}{2}\right)^\kappa}$ generalises to higher dimensions the fact that decompositions of Laurent polynomial matrices are at least approximately equivalent to parallel frequency-by-frequency decompositions of complex matrices.

\subsection{Quad-quaternion algebra}
Let $\setA$ be the quad-quaternion algebra $\setH \otimes \setH$.
\citet{gong_quad-quaternion_2008} reduces the problem of computing the eigenvalue decomposition of a Hermitian covariance matrix in $\setA^{n \times n}$ to the EVD of a Hermitian $2n \times 2n$ biquaternion matrix, and \citet{le_bihan_music_2007} reduces the problem of computing the EVD of a Hermitian $m \times m$ biquaternion matrix to the EVD of a Hermitian $2m \times 2m$ quaternion matrix, so that ultimately the EVD of a Hermitian $4n \times 4n$ quaternion matrix is required.

Using instead the approach described in Section~\ref{sec:semisimple}, we note that the algebra $\setH \otimes \setH$ is isomorphic to $\setR^{4 \times 4}$. An explicit isomorphism is given by
\begin{align*}
&1 \otimes 1 \mapsto \bI_4,\ \ 
\rmi \otimes 1 \mapsto \begin{pmatrix}
0 & -1 & 0 & 0\\
1 & 0 & 0 & 0\\
0 & 0 & 0 & -1\\
0 & 0 & 1 & 0
\end{pmatrix},\ \ 
\rmj \otimes 1 \mapsto \begin{pmatrix}
0 & 0 & -1 & 0\\
0 & 0 & 0 & 1\\
1 & 0 & 0 & 0\\
0 & -1 & 0 & 0
\end{pmatrix},\\
&1 \otimes \rmi \mapsto \begin{pmatrix}
0 & -1 & 0 & 0\\
1 & 0 & 0 & 0\\
0 & 0 & 0 & 1\\
0 & 0 & -1 & 0
\end{pmatrix},\ \ 
1 \otimes \rmj \mapsto \begin{pmatrix}
0 & 0 & -1 & 0\\
0 & 0 & 0 & -1\\
1 & 0 & 0 & 0\\
0 & 1 & 0 & 0
\end{pmatrix},
\end{align*}
where the remaining 11 basis elements can be obtained from products of these 5. This isomorphism can be obtained by identifying the quad-quaternion $a \otimes b$ with the linear transformation $q \mapsto aqb,\ q \in\setH$, and it is a $*$-algebra isomorphism. Because $\setH \otimes \setH$ is a twisted group algebra, this gives us a unitary basis for $\setR^{4 \times 4}$, and in particular the basis is orthogonal. The $\frac{1}{2}\|\bullet\|_F$ norm on the $\setR^{4 \times 4}$ representation is equal to the $\|\bullet\|_2$ norm on $\setH \otimes \setH$.

The EVD of a Hermitian matrix is equal to its SVD. The isomorphism above allows us to compute the $\setA$SVD of an $n \times n$ quad-quaternion matrix from the SVD of a $4 n \times 4 n$ real matrix. This is a more direct, simpler, and more computationally  efficient approach compared to using the EVD/SVD of a $4 n \times 4 n$ quaternion matrix.

Similarly, the algebra of biquaternions $\setH \otimes \setC$ is isomorphic to $\setC^{2 \times 2}$, and computing the SVD of a $2n \times 2n$ complex matrix is a more direct and efficient way of obtaining a biquaternion SVD than computing the SVD of a $2n \times 2n$ quaternion matrix as in \citet{le_bihan_music_2007}.


\subsection{Conformal geometric algebra $C\ell(4,1)$}\label{sec:CGA}
Consider the 32-dimensional conformal geometric algebra $\setA=C\ell(4,1)$ with standard ordered basis \[\begin{split}\mathcal{B}=
\{&e_1,\ldots,e_{32}\}\\
=\{&\gamma_0,\\
&\gamma_1,\gamma_2,\gamma_3,\gamma_+,\gamma_-,\\
&\gamma_1\gamma_2,\gamma_1\gamma_3,\gamma_1\gamma_+,\gamma_1\gamma_-,\gamma_2\gamma_3,\gamma_2\gamma_+,\gamma_2\gamma_-,\gamma_3\gamma_+,\gamma_3\gamma_-,\gamma_+\gamma_-,\\
&\gamma_1\gamma_2\gamma_3,\gamma_1\gamma_2\gamma_+,\gamma_1\gamma_2\gamma_-,\gamma_1\gamma_3\gamma_+,\gamma_1\gamma_3\gamma_-,\gamma_1\gamma_+\gamma_-,\\
&\gamma_2\gamma_3\gamma_+,\gamma_2\gamma_3\gamma_-,\gamma_2\gamma_+\gamma_-,\gamma_3\gamma_+\gamma_-,\\
&\gamma_1\gamma_2\gamma_3\gamma_+,\gamma_1\gamma_2\gamma_3\gamma_-,\gamma_1\gamma_2\gamma_+\gamma_-,\gamma_1\gamma_3\gamma_+\gamma_-,\gamma_2\gamma_3\gamma_+\gamma_-,\\
&\gamma_1\gamma_2\gamma_3\gamma_+\gamma_-\},\end{split}\]
where $\gamma_0=1$, $\gamma_1^2=\gamma_2^2=\gamma_3^2=\gamma_+^2=1$, $\gamma_-^2=-1$ and the grade-1 basis elements $\{\gamma_1,\gamma_2,\gamma_3,\gamma_+,\gamma_-\}$ anti-commute.

\citet[Theorem~2.5.2]{tian_universal_1998} describes an isomorphism between $C\ell(4,1)$ and $\setC^{4 \times 4}$. After correcting some typos in \citet[Equation~2.4.4]{tian_universal_1998}, the isomorphism is given by
\begin{align*}
& \gamma_0 \mapsto \bI_4,\ \ 
\gamma_1 \mapsto \begin{pmatrix}
1 & 0 & 0 & 0\\
0 & -1 & 0 & 0\\
0 & 0 & 1 & 0\\
0 & 0 & 0 & -1
\end{pmatrix},\ \ 
\gamma_2 \mapsto \begin{pmatrix}
0 & 1 & 0 & 0\\
1 & 0 & 0 & 0\\
0 & 0 & 0 & 1\\
0 & 0 & 1 & 0
\end{pmatrix},\\
&\gamma_3 \mapsto \begin{pmatrix}
0 & 0 & 0 & 1\\
0 & 0 & -1 & 0\\
0 & -1 & 0 & 0\\
1 & 0 & 0 & 0
\end{pmatrix},
\gamma_+ \mapsto \begin{pmatrix}
0 & 0 & 0 & -\rmi\\
0 & 0 & \rmi & 0\\
0 & -\rmi & 0 & 0\\
\rmi & 0 & 0 & 0
\end{pmatrix},
\gamma_- \mapsto \begin{pmatrix}
0 & -1 & 0 & 0\\
1 & 0 & 0 & 0\\
0 & 0 & 0 & 1\\
0 & 0 & -1 & 0
\end{pmatrix},
\end{align*}
where the remaining 28 basis elements can be obtained from products of these 6. This isomorphism is a $*$-algebra isomorphism (using the involution $\bar{e_i}=e_i^{-1}$).

Figure~\ref{fig:QR} shows a matrix $\bA \in C\ell(4,1)^{3 \times 2}$ whose coefficients are random standard Gaussian, along with its $C\ell(4,1)$QRD, which was computed both directly in $C\ell(4,1)$ using Algorithm~\ref{code:QR} with $\bar{e_i}=e_i^{-1}$, $\|\bullet\|=\|\bullet\|_\infty$, $\beta=\beta_\mathcal{B}$ and $\epsilon=10^{-16}$; and through the $\setC$QRD of the $\setC^{4 \times 4}$ representation using Algorithm~\ref{code:QR} with complex conjugation, $\|\bullet\|=\|\bullet\|_2$, $\beta=\beta_{\max}$ and $\epsilon=10^{-16}$ .
The first approach required 1658 Givens rotations and 2 sweeps. The second approach required 80 Givens rotations and 1 sweep. Note that we would normally expect the QR decomposition of a matrix in $\setC^{12 \times 8}$ to require 60 Givens rotations, since there are 60 entries below the diagonal, and this would be the case with standard $\setC$QR algorithms. However, our ``naive'' implementation of the $\setC$ Givens rotations may fail to set a coefficient exactly to 0 because of rounding error. Let $\|\bullet\|_2$ be the norm on $C\ell(4,1)$. After setting all entries of $\bR$ below the diagonal to 0, the reconstruction error $\|\bA-\bQ \bR\|_F$ is $3.39 \cdot 10^{-14}$ and $1.21 \cdot 10^{-14}$ for the first and second approach respectively.

Figure~\ref{fig:SVD} shows the $C\ell(4,1)$SVD of $\bA$. Again,  this was computed both directly in $C\ell(4,1)$ using Algorithm~\ref{code:SVD} with $\bar{e_i}=e_i^{-1}$, $\|\bullet\|=\|\bullet\|_\infty$, $\beta=\beta_\mathcal{B}$ and $\epsilon=10^{-16}$; and through the $\setC$SVD of the $\setC^{4 \times 4}$ representation using Algorithm~\ref{code:SVD} with complex conjugation, $\|\bullet\|=\|\bullet\|_2$, $\beta=\beta_{\max}$ and $\epsilon=10^{-16}$ .  The first approach required 209 QRDs and a total of 42935 Givens rotations. The second approach required 387 QRDs and a total of 2770 Givens rotations. Let $\|\bullet\|_2$ be the norm on $C\ell(4,1)$. After setting all off-diagonal entries of $\bD$ to 0, the reconstruction error $\|\bA-\bU \bD \bV^\herm\|_F$ is $8.51 \cdot 10^{-13}$ and $5.71 \cdot 10^{-14}$ for the first and second approach respectively.

When comparing the two approaches one should bear in mind that a $C\ell(4,1)$ Givens rotations (with $b=1$) requires 16 times more real additions and multiplications than a $\setC$ Givens rotation (with $b=1$), and also that the error tolerance $\epsilon$ used in each approach refers to a different norm, so that one should set the error tolerance to $\frac{\epsilon}{16}$ instead of $\epsilon$ in the second approach to ensure that $\max\limits_{i,j:i<j} \|r_{ij}\|_\infty\leq\epsilon$. Figure~\ref{fig:rotations} compensates for these two facts when comparing the computational complexity of the two approaches for varying $\epsilon$. It shows that the second approach is much more computationally efficient. 

\begin{figure}[htbp]
\centering
\includegraphics[width=0.43\textwidth]{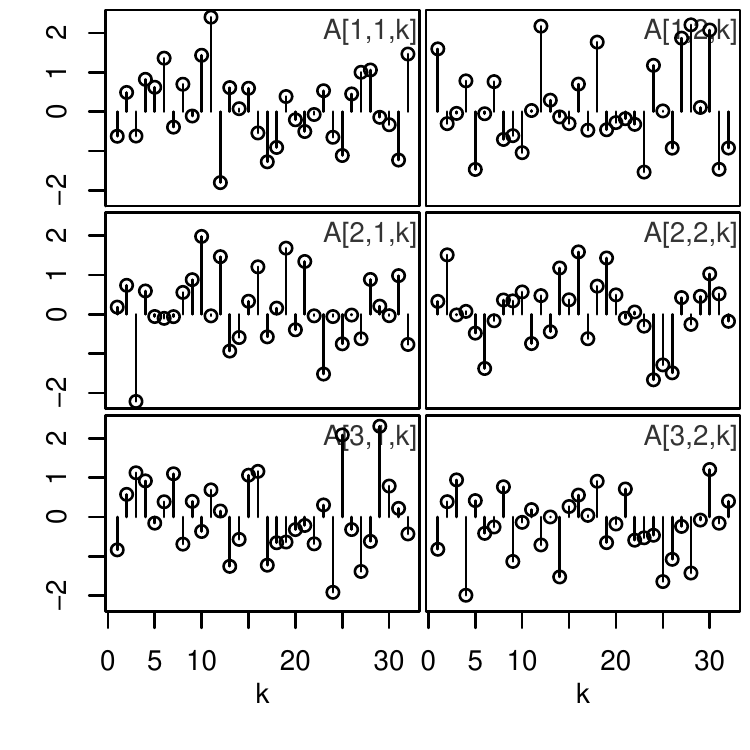}\\
\includegraphics[width=0.43\textwidth]{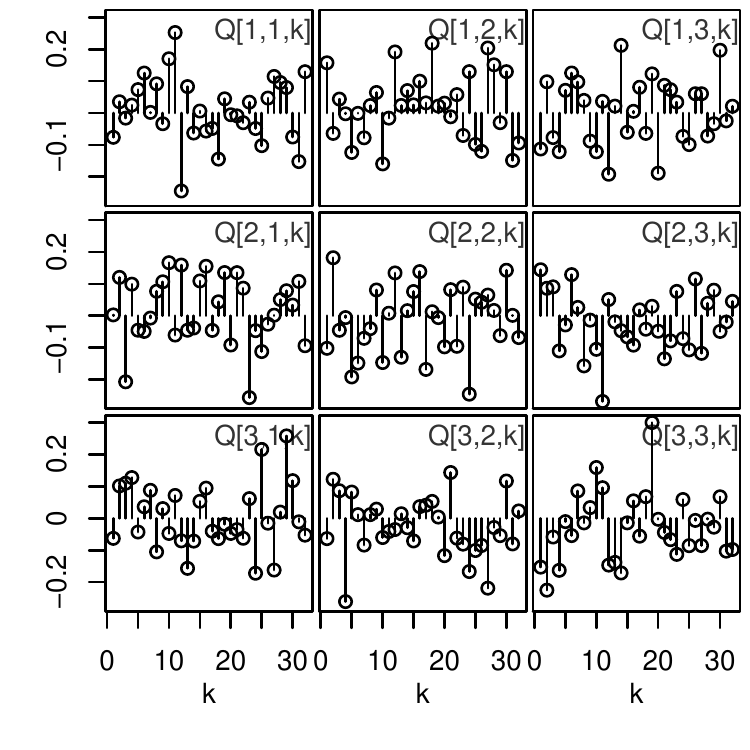}
\includegraphics[width=0.43\textwidth]{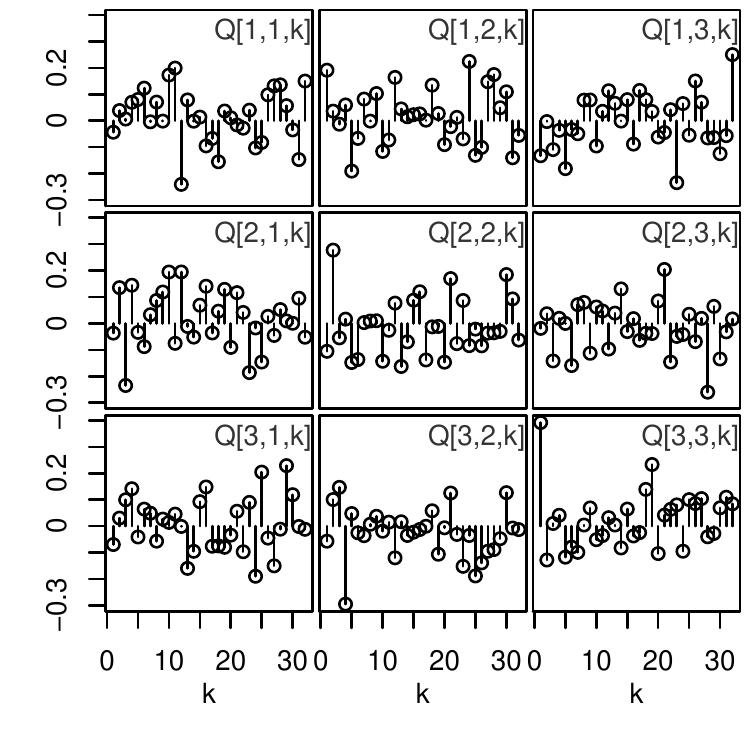}\\
\includegraphics[width=0.43\textwidth]{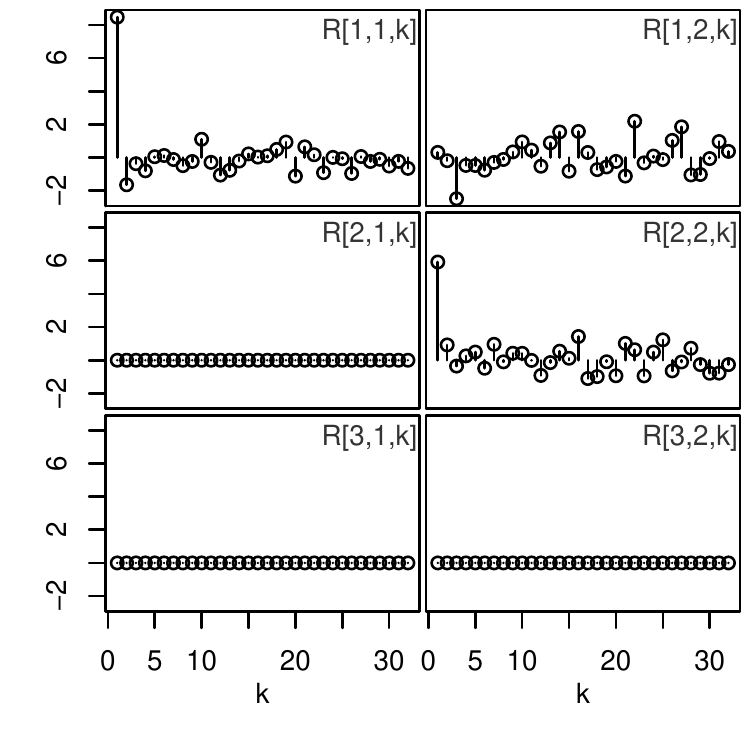}
\includegraphics[width=0.43\textwidth]{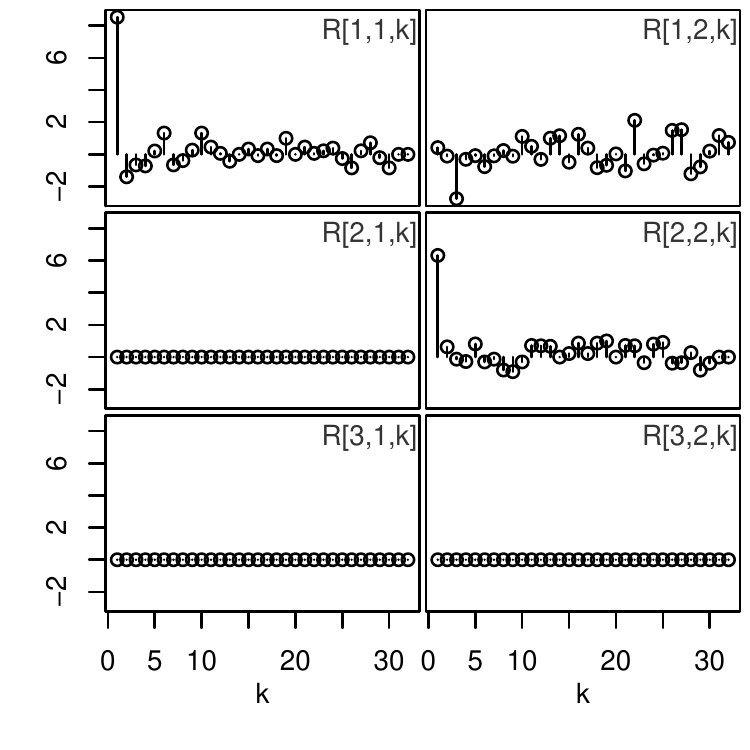}
\caption{\label{fig:QR}
\textbf{Top:} A Gaussian random matrix $\bA \in C\ell(4,1)^{3 \times 2}$.\\
\textbf{Left:} The $C\ell(4,1)$QR of $\bA$ computed using Algorithm~\ref{code:QR} with $\|\bullet\|_\infty$ and $\beta_\mathcal{B}$.\\
\textbf{Right:}  The $C\ell(4,1)$QR of $\bA$ computed using the approach of Section~\ref{sec:semisimple}, i.e. the $\setC$QR of a representation of $\bA$ in $\setC^{12 \times 8}$.}
\end{figure}
\begin{figure}[htbp]
\centering
\includegraphics[width=0.44\textwidth]{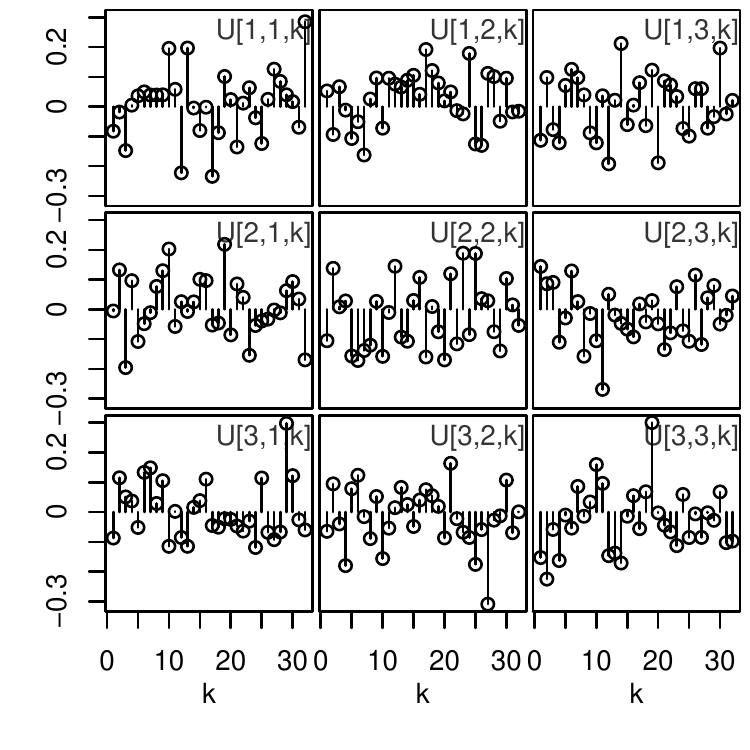}
\includegraphics[width=0.44\textwidth]{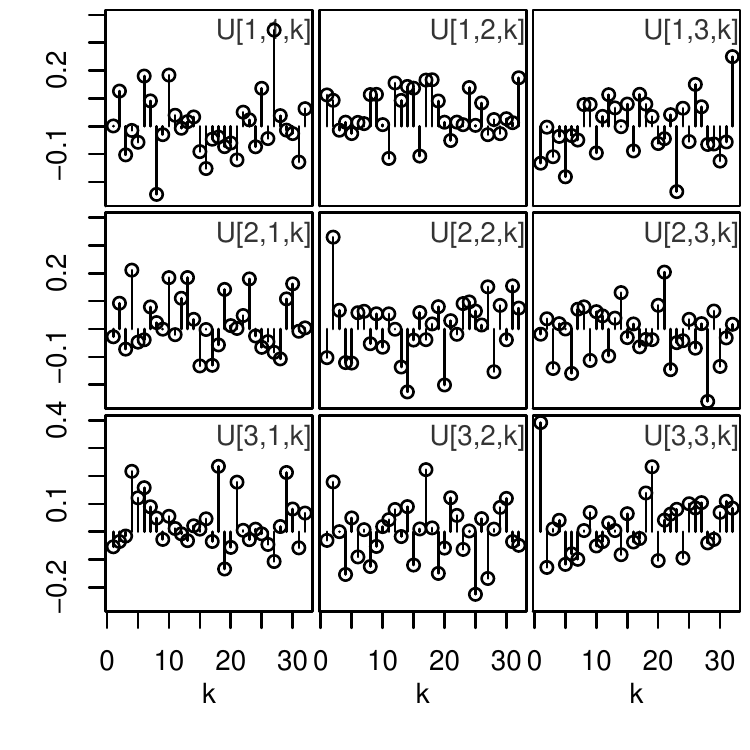}\\
\includegraphics[width=0.44\textwidth]{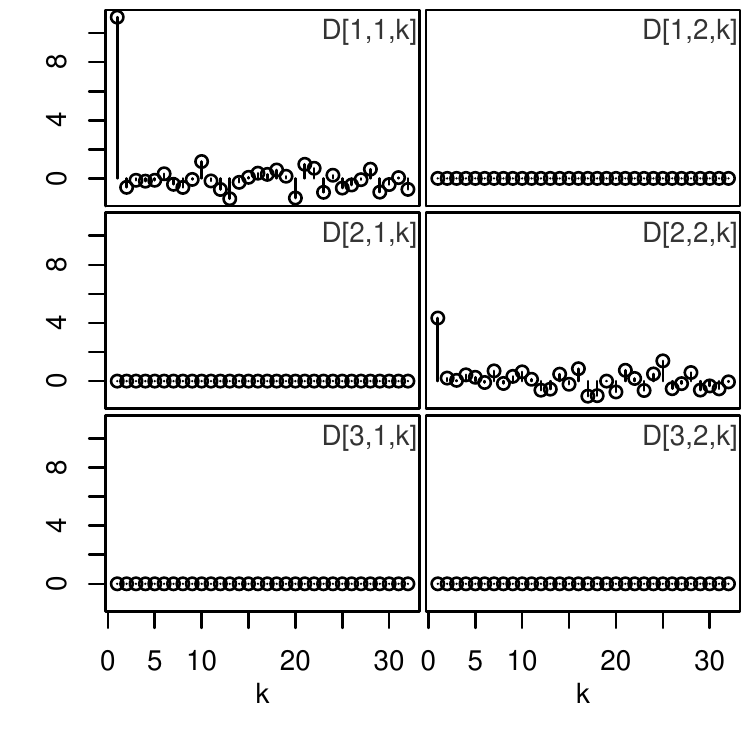}
\includegraphics[width=0.44\textwidth]{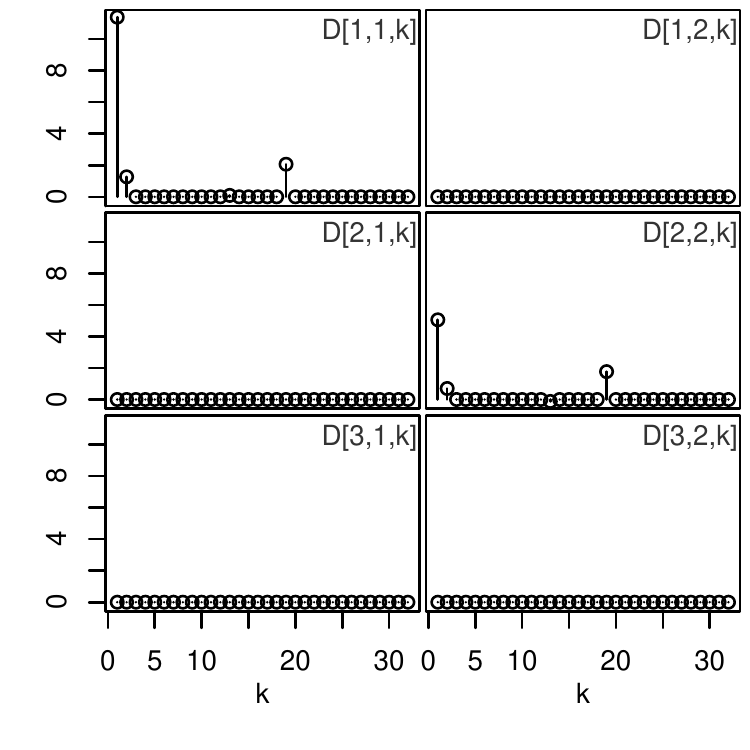}\\
\includegraphics[width=0.44\textwidth]{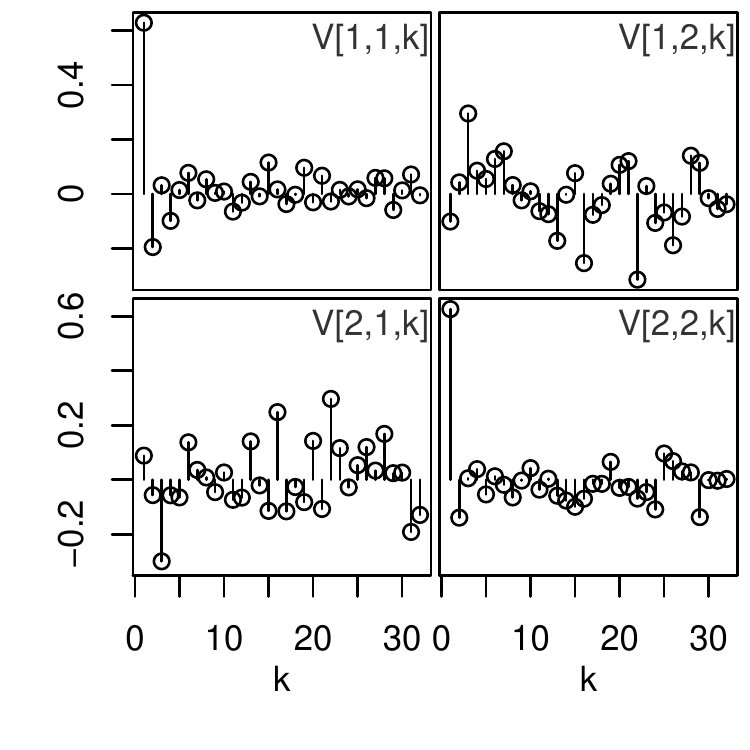}
\includegraphics[width=0.44\textwidth]{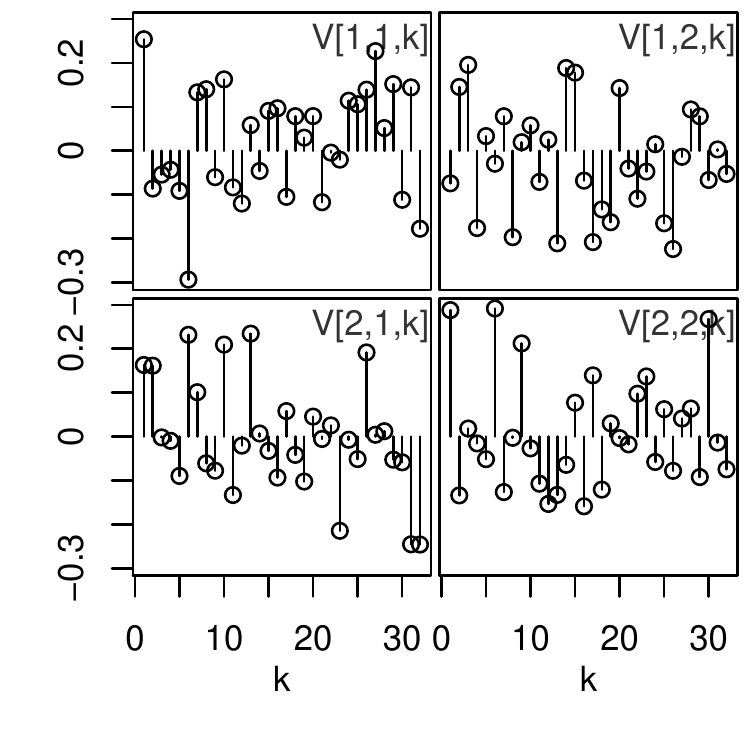}
\caption{\label{fig:SVD}
\textbf{Left:} The $C\ell(4,1)$SVD of $\bA$ computed using Algorithm~\ref{code:SVD} with $\|\bullet\|_\infty$ and $\beta_\mathcal{B}$.\\
\textbf{Right:}  The $C\ell(4,1)$SVD of the same matrix $\bA$ computed using the approach of Section~\ref{sec:semisimple}, i.e. the $\setC$SVD of a representation of $\bA$ in $\setC^{12 \times 8}$.}
\end{figure}
\begin{figure}[htbp]
\centering
\includegraphics[width=0.44\textwidth]{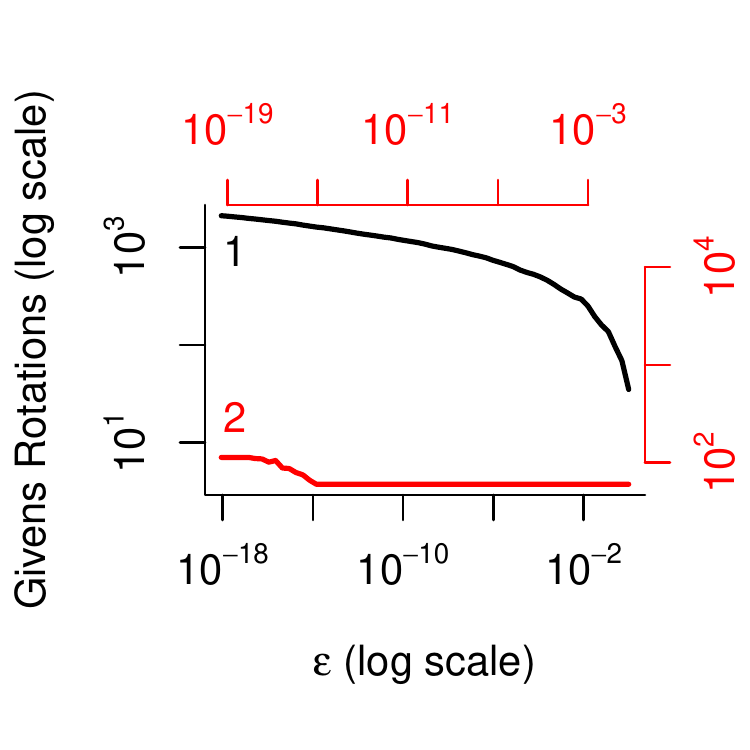}
\includegraphics[width=0.44\textwidth]{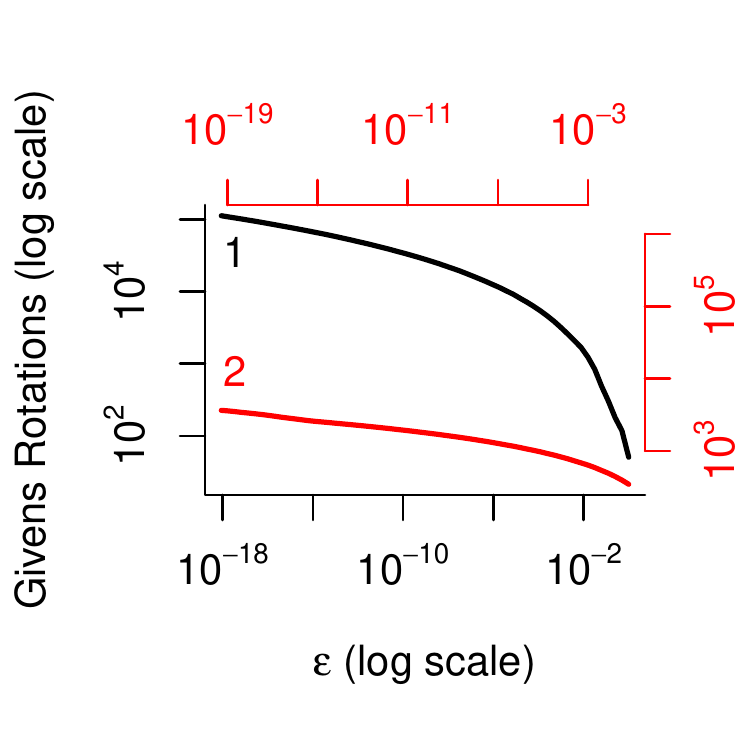}
\caption{\label{fig:rotations}
The number of Givens rotations needed to decompose $\bA \in C\ell(4,1)^{3 \times 2}$ as a function of the error tolerance $\epsilon$.\\
\textbf{Left:} Curve 1 corresponds to the $C\ell(4,1)$QRD computed using Algorithm~\ref{code:QR} with $\|\bullet\|_\infty$ and $\beta_\mathcal{B}$. Curve 2 corresponds to the $C\ell(4,1)$QRD computed using the approach of Section~\ref{sec:semisimple}.\\
\textbf{Right:} Curve 1 corresponds to the $C\ell(4,1)$SVD computed using Algorithm~\ref{code:SVD} with $\|\bullet\|_\infty$ and $\beta_\mathcal{B}$. Curve 2 corresponds to the $C\ell(4,1)$SVD computed using the approach of Section~\ref{sec:semisimple}.\\
Each $C\ell(4,1)$ Givens rotation involves about 16 times more operations than each $\setC$ Givens rotation. Also, when using the approach of Section~\ref{sec:semisimple} the error tolerance should be set to $\frac{\epsilon}{16}$ if one wishes to ensure that all coefficients below the diagonal are less than $\epsilon$ when transforming back to $C\ell(4,1)$. Hence for a fair comparison, Curve 2 is plotted on a different scale, corresponding to the top and right axes.
}
\end{figure}

One of the features of the $C\ell(4,1)$SVD obtained with the approach of Section~\ref{sec:semisimple} is that because the diagonal elements of $\bD$ must have a real diagonal representation in $\setC^{4 \times 4}$, they are restricted to lie in a 4-dimensional subspace of $C\ell(4,1)$, namely the subalgebra spanned by $\{\gamma_0,\gamma_1,\gamma_2\gamma_-,\gamma_1\gamma_2\gamma_-\}$. This feature is noticeable in the middle right subplot of Figure~\ref{fig:SVD}.
Similarly, because the diagonal elements of $\bR$ in a $C\ell(4,1)$QR obtained using the approach of Section~\ref{sec:semisimple} must have a representation which is upper triangular with real diagonal entries in $\setC^{4 \times 4}$, they are restricted to lie in a 10-dimensional subspace of $C\ell(4,1)$.

\section{Conclusion}
The computation of matrix decompositions such as the QRD, SVD, or symmetric EVD for matrices over an algebra $\setA$ has so far required the development of bespoke algorithms for each new algebra considered. We have described two general approaches to compute these matrix decompositions. The first approach generalises standard real, complex, quaternion and polynomial QR and SVD algorithms. It can be easily applied to (twisted) group algebras and in particular Clifford algebras in their standard basis. The second approach uses a representation of $\setA$ which reduces the matrix to a Kronecker sum of real, complex and quaternion matrices, which can then be decomposed in parallel. Although the Artin-Wedderburn theorem guarantees that this latter method is applicable to any finite-dimensional semi-simple algebra, finding the representation explicitly may not be straightforward for algebras other than $C\ell(p,q)$. The number of operations required for the first approach typically grows to infinity as the error tolerance $\epsilon$ tends to $0$, whereas this is not true of the second approach which is typically more computationally efficient. Another source of computationally efficiency for the second approach is that it can rely on more sophisticated $\setR/\setC/\setH$SVD algorithms, such as using Householder transformations \citep{sangwine_quaternion_2006}.

Although the approaches described here are general enough to cover the wide range of algebras used in signal processing, they do not apply to all real algebras. For example, the algebra of upper triangular matrices in $\setR^{m \times m}$ ($m>1$) is not semi-simple and does not admit any norm with a decent $\beta$.

\FloatBarrier
\section*{References}
\bibliographystyle{elsarticle-num} 
\bibliography{biblio}

\end{document}